\documentclass{amsart}

\numberwithin{equation}{section}

\usepackage{amssymb}
\usepackage{enumerate, xspace}
\usepackage[colorlinks]{hyperref}

\usepackage{color}

\newtheorem{thm}{Theorem}[section]
\newtheorem{lem}[thm]{Lemma}

\newtheorem{rem}[thm]{Remark}

\newcommand\cB{{\mathcal B}}

\newcommand\cD{{\mathcal D}}

\newcommand\bR{{\mathbb R}}

\newcommand\bZ{{\mathbb Z}}

\newcommand\ve{\varepsilon}

\newcommand\vf{\varphi}

\newcommand{\Id}{1}

\newcommand{\pd}[1]{\partial_{#1}}
\newcommand{\pr}{^*}
\newcommand{\boundary}{\partial}

\begin{document}

\title[Map lattices with collisions]{Map Lattices coupled by collisions}
\author{Gerhard Keller and Carlangelo Liverani}
\address{Gerhard Keller\\Department Mathematik\\Universit\"at
  Erlangen-N\"urnberg\\Bismarckstr. 1$\frac12$, 91052 Erlangen, Germany}
\email{{\tt keller@mi.uni-erlangen.de}}
\address{Carlangelo Liverani\\
Dipartimento di Matematica\\
II Universit\`{a} di Roma (Tor Vergata)\\
Via della Ricerca Scientifica, 00133 Roma, Italy.}
\email{{\tt liverani@mat.uniroma2.it}}
\date{\today}
\begin{abstract}
  We introduce a new coupled map lattice model in which the weak interaction
  takes place via rare ``collisions''. By ``collision'' we mean a strong
  (possibly discontinuous) change in the system. For such models we prove
  uniqueness of the SRB measure and exponential space-time decay of
  correlations.
\end{abstract}
\thanks{We are indebted to the ESI where, during the Workshop on Hyperbolic
  Dynamical Systems with Singularities (2008), this work was started. L.C. thanks the ENS, Paris, where he was invited during part of this work. Also we like to thank the Institut Henri Poincare - Centre Emile Borel where, during the trimester M\'ecanique statistique, probabilit\'es et syst\`emes de particules (2008), this work was finished. Finally, G.K. acknowledges the support by a grant from the DFG}
\maketitle
\section{Introduction}
During the last years many results have been published on coupled map
lattices. Much of this work deals with the weak coupling situation; see
\cite{book} and references therein to have a more precise idea of all the
related work and results. In most of this considerable body of work the weak
coupling is described by a \emph{diffeomorphism} of the state space close to
identity. Only recently couplings close to identity but not diffeomorphic
could be investigated in a mathematically rigorous way. This setting (coupling
close to identity) models a weak interaction and is reminiscent of a situation
in which a collection of systems is interacting via a weak potential. Yet, a
collection of systems can also have interactions that are weak only on
average. The typical example of such a situation are interactions that can be
strong but are rare, such as in the case of collisions in a rarefied gas.
Examples of such a situation that has attracted some attention lately are
models of the type introduced in \cite{liverani-bunimovich-p-s} and recently
used in \cite{gaspard-gilbert1, gaspard-gilbert2} to argue for a derivation of
the Fourier Law from microscopic dynamics.

In this note we consider a simple system of coupled dynamics of the latter
type. Namely a lattice of piecewise expanding interval maps with strong but
rare interactions.  Of course, this model is very far from a system of
interacting disks, yet it is interesting that the available techniques can be
applied to a case with strong but rare interactions. Given the present efforts
in trying to devise a setting for Anosov maps with discontinuities that shares
the same good properties of the $BV$ setting for piecewise expanding maps
\cite{demer-liverani, baladi-gouezel} it is reasonable to hope that in the
future the present results could be extend to more realistic situations.

What we prove is that such a system, if the local maps are sufficiently
expanding, has a unique SRB measure with exponential mixing properties in time
and space. These are the same type of results we proved for systems with weak interactions of the
``coupling close to identity'' type in \cite{kl-cmp}.

\section{The system and the results}

Consider $X=[0,1]^{\bZ^d}$ and a map $T:[0,1]\to [0,1]$, piecewise smooth and
expanding, that is $|DT|\geq \lambda>1$. The uncoupled system is described by
the product dynamics $F_0:X\to X$ defined by $F_0(x)_i=T(x_i)$. The strength
of the interaction will be expressed by a parameter $\ve>0$ which measures how
rare the interactions are. To be precise, let $\{e_i\}$ be the standard base
of $\bR^d$ and $V:=\{e_i, -e_i\}\subset \bR^d$, and, for each $\ve>0$
sufficiently small, fix a set of disjoint open intervals $\{A_{\ve,v}\}_{v\in
  V}\subset [0,1]$ such that $|A_{\ve,v}|=\ve$. Consider the
coupling\footnote{This special case is considered for simplicity. Yet, one can
  easily treat the more general case in which $x_{i+v}$ is replaced by
  $\phi_v(x_i, x_{i+v})$, for some set of invertible smooth maps
  $\phi_v:[0,1]^2\to[0,1]^2$.}
\[
(\Phi_\ve(x))_i=\begin{cases} x_{i+v}\quad &\text{ if } x_i\in A_{\ve,v} \text{ and } x_{i+v}\in A_{\ve,-v} \text{ for some }v\in V\\
  x_i &\text { otherwise.}
                                               \end{cases}
\] 
The dynamics we are interested in is then defined by $F_\ve=\Phi_\ve\circ F_0$.
\begin{rem}
  The interpretation of $\Phi_\ve$ is quite obvious: nearby systems can
  interact only if the coordinates belong to a small set (the ``collision''
  set). If this is the case, then the coordinates undergo a violent change.
\end{rem}  
 
We use $\|\cdot\|$ for the bounded variation norm and $|\cdot|$ for the total
variation norm of a measure. That is, calling $\cD$ the space of
differentiable local functions on $X$ and
$\cD_1:=\{\vf\in\cD\;:\;|\vf|_\infty\leq 1\}$, for each complex Borel measure
$\mu$ on $X$ we define
\begin{equation}\label{norms}
  \begin{split}
    |\mu|&:=\sup_{\vf\in\cD_1}\mu(\vf)\\
    \|\mu\|&:=\sup_{i\in\bZ^d}\sup_{\vf\in\cD_1}\mu(\partial_{x_i}\vf).
  \end{split}
\end{equation}
To study the statistical properties of the map $F_\ve$ we will follow the well
established path of studying its action on the space of measures. That is we
will investigate the operator $F_\ve\pr$.\footnote{As usual for each map
  $\Psi:X\to X$ the operator $\Psi\pr$ is defined by
  $\Psi\pr\mu(\vf)=\mu(\vf\circ \Psi)$, for each measure $\mu$ and function
  $\vf$.}  To this end we will follow the path laid down in \cite{kl-cmp} and
restrict our study to measures belonging to the space
$\cB:=\{\|\mu\|<\infty\}$, see \cite{kl-lecture} for a detailed explanation of
the meaning and properties of this choice. We will then use exactly the same
strategy as developed in \cite{kl-cmp}.

The following lemma is proven in \cite[Proposition 4]{kl-lecture}; see also
Lemma 2.2 of \cite{kl-cmp}.  
\begin{lem}\label{lem:indip}
There exists a constant $B_0>0$ such that for each $\mu\in\cB$ holds
\[
\begin{split}
|F_0\pr\mu|&\leq |\mu|\\
\|F_0\pr\mu\|&\leq  2\lambda^{-1}\|\mu\|+B_0|\mu|.
\end{split}
\]
\end{lem}
The proof of the next lemma is provided in section \ref{sec:proofs}.
\begin{lem} \label{lem:couple}
There exists a constant $B_1>0$ such that for each $\mu\in\cB$ holds
\[
\begin{split}
|\Phi_\ve\pr\mu|&\leq |\mu|\\
\|\Phi_\ve\pr\mu\|&\leq  (2+2d)\|\mu\|+B_1\ell_\ve^{-1}|\mu|.
\end{split}
\]
where $\ell_\ve$ is the minimal distance between the intervals $A_{\ve,v}$
$(v\in V)$. 
\end{lem}
The above lemmas imply a Lasota-Yorke type inequality for $F_\ve$, namely:
there exists a constant $B>0$ such that
\begin{equation}\label{eq:lasota-yorke}
\begin{split}
|F_\ve\pr\mu|&\leq |\mu|\\
\|F_\ve\pr\mu\|&\leq  (4+4d)\lambda^{-1}\|\mu\|+B\ell_\ve^{-1}|\mu|.
\end{split}
\end{equation}  
The second of these inequalities is relevant in the case
$\sigma=(4+4d)\lambda^{-1}<1$ which we will assume from now on.\footnote{With
  some more work one can certainly 
    weaken the requirement on the expansion
  constant $\lambda$. Since the goal of this note is to show that this type of
  maps can be treated by transfer operator methods, we decided to restrict
  ourselves
 to the simplest possible setting.}  The only thing left is to
define a one-site decoupling $\Phi_{\ve,i}$ of the dynamics and to show that we
can decouple at a single site paying only a small price. Namely, we define
\begin{equation}\label{eq:decouple}
  (\Phi_{\ve,i}(x))_k
  =
  \begin{cases}
    (\Phi_\ve(x))_k\quad &\text{ if }k-i\not\in V\cup\{0\}\\
    (\Phi_\ve(x))_k &\text{ if }  
v:=k-i  
   \in V,\text{ and }x_k\not\in A_{\ve,-v}\\
    x_k &\text { otherwise.}
  \end{cases}    
\end{equation}
A moment of thought will show that $\Phi_{\ve,i}$ differs from $\Phi_\ve$ only
insofar the coordinate $x_i$ is now independent of the other coordinates.
\begin{lem}\label{lem:decouple}
  For each $i\in\bZ^d$,
  \[
  \begin{split}
    |\Phi_{\ve,i}\pr\mu|&\leq |\mu|\\
    \|\Phi_{\ve,i}\pr\mu\|&\leq  (2+2d)\|\mu\|+B_1\ell_\ve^{-1}|\mu|\\    
    |\Phi_\ve\pr\mu-\Phi_{\ve,i}\pr\mu|&\leq 4d\ve\, \|\mu\|.
  \end{split}
  \]
\end{lem}
\noindent The proof of this lemma can be found in Section \ref{sec:proofs}.

Equation \eqref{eq:lasota-yorke} and Lemma \ref{lem:decouple} are all that is
needed to prove our main theorem.
\begin{thm}\label{thm:main}
    If $\lambda>4+4d$ and if $\liminf_{\ve\searrow0}\ell_\ve>0$, then there
    is some $\ve_0>0$ such that for each $\ve\in(0,\ve_0)$ there exists a
    unique invariant element
  $\mu_{inv}\in\cB$ for the dynamics $F_\ve$. In addition, for all local
  smooth observables, $\mu_{inv}$ enjoys exponential decay of space-time
  correlations. 
\end{thm}
\begin{proof}
  The proof follows verbatim the arguments in \cite{kl-cmp}. More precisely,
  Lemma~2.3 of \cite{kl-cmp} holds in the present context under the above
  hypotheses. The proof of such a lemma is  exactly the same as in Section~3 of \cite{kl-cmp}, provided:
  \begin{enumerate}[a)]
  \item Lemma~2.2 of
  \cite{kl-cmp} is replaced by inequalities \eqref{eq:lasota-yorke} above; 
\item instead of inequality (3.16) of \cite{kl-cmp} (and of Lemma~3.2 in
  \cite{kl-cmp} from which that estimate follows) one uses the last assertion
  of our Lemma~\ref{lem:decouple};
\item   the analog of equation (3.17) in \cite{kl-cmp} is now obtained using
    estimate \eqref{eq:lasota-yorke} above instead of Lemma~2.2 of
    \cite{kl-cmp}. 
 \end{enumerate}
 The proof of the theorem is then exactly the same as the proof of Theorem~2.1
 of \cite{kl-cmp} (see section 4 of the cited paper).
\end{proof}
\section{Proofs}\label{sec:proofs}
\begin{proof}[Proof of Lemma \ref{lem:couple}]
  Let $\vf$ be a smooth local function and let $i\in\bZ^d$.
  Then, setting $\Delta_{\ve,v,i}=\{x\;:\; x_i\in
  A_{\ve,v}, x_{i+v}\in A_{\ve,-v}\}$ and $\Delta_{\ve,i}=\cup_{v\in V}
  \Delta_{\ve,v,i}$, we have
  \begin{displaymath}
    (\pd{x_i}\vf)\circ\Phi_\ve(x)=
    \begin{cases}
      \pd{x_i}\left((\vf\circ\Phi_\ve)\right)(x)&\text{if
      }x\in\Delta_{\ve,i}^c\\
      \pd{x_{i+v}}\left((\vf\circ\Phi_\ve)\right)(x)&\text{if
      }x\in\Delta_{\ve,v,i} \,.
    \end{cases}
  \end{displaymath}
  Therefore, for $\mu\in \cB$,
  \begin{equation}
    \label{eq:LY1}
    \begin{split}
      \Phi_\ve\pr\mu(\pd{x_i}\vf) &=
      \int_X  (\pd{x_i}\vf)\circ\Phi_\ve\,d\mu\\
      &= \int_{\Delta_{\ve,i}^c} \pd{x_i}(\vf\circ\Phi_{\ve})\, d\mu +
      \sum_{v\in
        V}\int_{\Delta_{\ve,v,i}} \pd{x_{i+v}}(\vf\circ\Phi_\ve)\,d\mu.
    \end{split}
  \end{equation}
  In order to estimate these integrals against the variation of $\mu$ in
  directions $i$ and $i+v$ we must modify the test function
  $\vf\circ\Phi_\ve$ in such a way that it becomes continuous in $x_i$ or
  $x_{i+v}$, respectively (see the characterization of $\|\cdot\|$ in \cite{kl-lecture} section
3.3).

  For the first integral, let $\vf_*$ be a function that, for each fixed
  $x_{\neq i}$, is piecewise linear (but not necessarily continuous!) in $x_i$
  interpolating between the requirements $\vf_*|_{\Delta_{\ve,i}}=0$ and
  $\vf_*(x)=\vf(\Phi_\ve(x))$ for each $x=(x_{\neq
    i},x_i)\in\boundary_i\Delta_{\ve,i}$ where $\boundary_i\Delta_{\ve,i}$
  denotes the set of those points in the boundary of $\Delta_{\ve,i}$ where
  the boundary is normal to the $x_i$-direction. Thus, given $x_{\neq i}$, the
  partial derivative $\pd{x_i}\vf_*(x)$ exists for Lebesgue-a.e. $x_i$, and
  $\tilde\vf:=\Id_{\Delta_{\ve,i}^c}\cdot(\vf\circ \Phi_\ve)-\vf_*$ is
  Lipschitz continuous in $x_i$, $|\tilde\vf|_\infty\leq2|\vf|_\infty$,
  $\tilde\vf|_{\partial_i\Delta_{\ve,i}}=0$, and $|\pd{x_i}\vf_*|_\infty\leq
  C\ell_\ve^{-1}|\vf|_\infty$ where $\ell_\ve$ is the minimal distance between
  intervals $A_{\ve,v}$ ($v\in V$). Hence,
  \begin{displaymath}
    \int_{\Delta_{\ve,i}^c} \pd{x_i}(\vf\circ\Phi_{\ve})\, d\mu
    =
    \int_X \pd{x_i}\tilde\vf\,d\mu+\int_X \pd{x_i}\vf_*\, d\mu
    \leq
    2|\vf|_\infty\|\mu\|+C\ell_\ve^{-1}|\vf|_\infty\,|\mu|\ .
  \end{displaymath}
  
  Similarly, for each $v\in V$, let $\vf_v$ be a function that, for each fixed
  $x_{\neq i+v}$, is piecewise constant in $x_{i+v}$ and such that
  $\vf_v|_{\Delta_{\ve,i,v}}=0$ and $\vf_v(x)=\vf(\Phi_\ve(x))$ for each
  $x=(x_{\neq i+v},x_{i+v})\in\boundary_{i+v}\Delta_{\ve,i,v}$. Then
  $\tilde\vf_v:=\Id_{\Delta_{\ve,i,v}}   \cdot(\vf\circ\Phi_\ve) +\vf_v$ is Lipschitz continuous in
  $x_{i+v}$, $|\tilde\vf_v|_\infty\leq|\vf|_\infty$ and  $\pd{x_{i+v}}\vf_v(x)=0$  for 
  $x\not\in\boundary_{i+v}\Delta_{\ve,i,v}$. Hence
  \begin{displaymath}
    \int_{\Delta_{\ve,v,i}} \pd{x_{i+v}}(\vf\circ\Phi_\ve)\,d\mu
    =
    \int_X\pd{x_{i+v}}\tilde\vf_v\,d\mu-\int_X\pd{x_{i+v}}\vf_v\,d\mu
    \leq
    |\vf|_\infty\|\mu\|
  \end{displaymath}

  Observing equation \eqref{eq:LY1}, this yields
  \begin{displaymath}
    \|\Phi_\ve\pr\mu\|\leq(2+2d)\|\mu\|+C\ell_\ve^{-1}|\mu|\ .
  \end{displaymath}
  This proves the second inequality of the lemma. The first one is trivial.
\end{proof}
\begin{proof}[Proof of Lemma \ref{lem:decouple}]
  The first two inequalities of this lemma are proved exactly as in
  Lemma~\ref{lem:couple}. We turn to the third one.  Let $A_\ve=\cup_{v\in
    V}A_{\ve,v}$ and notice that $\Phi_\ve(x)=\Phi_{\ve,i}(x)$ for all $x\in
  X$ such that $x_i\not\in A_\ve$. Thus, for all smooth local functions $\vf$
  and all $\mu\in\cB$
  \[
  |\Phi_{\ve}\pr\mu(\vf)-\Phi_{\ve,i}\pr\mu(\vf)|=\sum_{v\in
    V}\left|\int_X\Id_{A_{\ve,v}}(x_i)\Id_{A_{\ve,-v}}(x_{i+v})\left[\vf\circ
      \Phi_\ve(x)-\vf\circ \Phi_{\ve,i}(x)\right]\mu(dx)\right|
  \]
  Next, since $\vf\circ \Phi_\ve$ depends only on finitely many variables, say the variables
  in the finite set $\Lambda_0\subset \bZ^d$, we can consider the
  marginal of $\mu$ on $[0,1]^\Lambda$, $\Lambda=\Lambda_0\cup\{i+v\}_{v\in V\cup\{0\}}$. Such a marginal is absolutely
  continuous with respect to Lebesgue measure, and its density $h$ is a
  function of bounded variation with $|h|_{BV}\leq \|\mu\|$, (see
  \cite{kl-lecture} for details). Hence,
  \[
  |\Phi_{\ve}\pr\mu(\vf)-\Phi_{\ve,i}\pr\mu(\vf)|\leq \sum_{v\in V}2|\vf|_\infty\int_{[0,1]^\Lambda} \Id_{A_{\ve,v}}(x_i)\Id_{A_{\ve,-v}}(x_{i+v})|h(x)|dx\\
  \]
  and, since $|(|h|)|_{BV}\leq |h|_{BV}$ (see e.g. the proof of Lemma 2.3 in
  \cite{BGK-07}), we can consider the marginal $h_v(x_i,x_{i+v}):=\int
  |h(x)|dx_{k\not\in\{i,i+v\}}$. As $|h_v|_{BV}\leq |(|h|)|_{BV}\leq |h|_{BV}\leq\|\mu\|$, the usual
  Sobolev inequalities imply
  \[
  \begin{split}
    |\Phi_{\ve}\pr\mu(\vf)-\Phi_{\ve,i}\pr\mu(\vf)|
    &\leq 
    \sum_{v\in V}2|\vf|_\infty\int_{[0,1]^2} \Id_{A_{\ve,v}}(x)\Id_{A_{\ve,-v}}(y)h_v(x,y)dxdy\\
    &\leq 2(2d)|\vf|_\infty\|\mu\|\ve.
  \end{split}
  \]
\end{proof}

\end{document}